\documentclass[12pt,a4paper,reqno]{amsart}

\usepackage[margin=3cm]{geometry}
\usepackage{graphics, amsmath,enumitem, comment}
\usepackage{graphicx}
\usepackage{epsfig}

\newif\ifcolorcomments
\newcommand{\allowcomments}[4]{
\newcommand{#1}[1]{\ifdraft{\ifcolorcomments{\textcolor{#4}{##1 --#3}}\else{\textsl{ ##1 \ --#3}}\fi}\else{}\fi}
}

\allowcomments{\commumtaz}{MH}{Mumtaz}{green}
\allowcomments{\comwang}{BW}{BWWang}{blue}
\allowcomments{\comkle}{DK}{DK}{magenta}
\allowcomments{\comnick}{NW}{Nick}{red}

\colorcommentstrue
\usepackage{amssymb}
\usepackage{amsmath,amsthm}
\usepackage{mathrsfs}
\usepackage{bbm}
\def\bc{\begin{center}}
\def\ec{\end{center}}
\def\be{\begin{equation}}
\def\ee{\end{equation}}

\def\N{\mathbb N}
\def\Z{\mathbb Z}
\def\Q{\mathbb Q}

\def\R{\mathbb R}

\def\H{\mathcal H}

\newcommand\hdim{\dim_{\mathrm H}}

\def\Z{\mathbb Z}

\newtheorem{lem}{Lemma}[section]
\newtheorem{prob}[lem]{Problem}

\newtheorem{pro}[lem]{Proposition}
\newtheorem{thm}[lem]{Theorem}

\newtheorem{exa}[lem]{Example}
\newtheorem{cor}[lem]{Corollary}
\newtheorem{rem}[lem]{Remark}
\numberwithin{equation}{section}
\usepackage{colortbl}

\newif\ifdraft\drafttrue

\newcommand{\subscript}[2]{$#1 _ #2$}

\begin{document}



\subjclass[2010]  {}

\title[Hausdorff measure  and Dirichlet non-improvable numbers]{Hausdorff measure  of sets of Dirichlet non-improvable numbers}

\author[M. Hussain]{M.~Hussain}
\address{Mumtaz Hussain, Department of Mathematics and Statistics, La Trobe University, PoBox199, Bendigo 3552, Australia. }\email{m.hussain@latrobe.edu.au}
\author[D. Kleinbock]{D.~Kleinbock}
\address {Dmitry Kleinbock, Brandeis University, Waltham MA 02454-9110.} \email{kleinboc@brandeis.edu}
\author[N. Wadleigh]{N.~Wadleigh}
\address{Nick Wadleigh, Brandeis University, Waltham MA 02454-9110.}  \email{wadleigh@brandeis.edu}
\author[B-W. Wang]{B-W.~Wang}
\address{Bao-wei Wang, School  of  Mathematics  and  Statistics,  Huazhong  University  ofScience  and  Technology, 430074 Wuhan, China}
 \email{bwei\_wang@hust.edu.cn}

\thanks{The research of  M.~Hussain is supported by the Endeavour Fellowship,  of D.\ Kleinbock  by NSF grant  
DMS-1600814, and of B-W.~Wang by NSFC of China (No.\ 11471130 and NCET-13-0236).}

\maketitle

\begin{abstract}

Let $\psi:\mathbb R_+\to\mathbb R_+$ be a non-increasing function. A real number $x$ is said to be $\psi$-Dirichlet improvable
if it admits an improvement to Dirichlet's theorem in the following sense:  
the system $$|qx-p|< \, \psi(t) \ \ {\text{and}} \ \ |q|<t$$
has a non-trivial integer solution for all large enough $t$. Denote the collection of such points by $D(\psi)$. In this paper we prove that the Hausdorff measure of the complement $D(\psi)^c$ (the set of  $\psi$-Dirichlet non-improvable numbers) obeys a zero-infinity law for a large class of dimension functions.
 Together with the Lebesgue measure-theoretic results established by Kleinbock \& Wadleigh (2016), our results contribute to building a complete metric theory for the
 set of Dirichlet non-improvable numbers.

\end{abstract}

\section{ Introduction}

\noindent At its most fundamental level, the theory of Diophantine approximation is concerned with the question of how well a real number can be approximated by rationals. A qualitative answer is provided by the fact that the set of rational numbers is dense in the real numbers. Seeking a quantitative answer leads to the theory of metric Diophantine approximation. Dirichlet's theorem (1842) is the starting point in this theory.

To simplify the presentation, we start by fixing some notation. We use $a\gg b$ to indicate that $|a/b|$ is sufficiently large, and $a\asymp b$ to indicate that $|a/b|$ is bounded between unspecified positive constants. We use $\lambda(\cdot)$,  $\hdim$ and $\mathcal{H}^s$ to denote the Lebesgue measure, Hausdorff dimension, and $s$-dimensional Hausdorff measure, respectively. We use `i.m.' for {\em `infinitely many'.}

\subsection{Diophantine approximation: improving Dirichlet's theorem }

\begin{thm}[Dirichlet 1842]\label{Dir}
\label{Dirichletsv}
 \noindent Given $x\in \R$ and $t>1
$, there exist integers $p,q$
 such that
  \begin{equation*}\label{eqdir} \left\vert qx-p\right\vert\leq 1/t \quad{\rm and} \quad  1\leq{q}<{t}. \end{equation*}

\end{thm}

An easy consequence (known before Dirichlet) is the following global statement concerning the `rate' of rational approximation to any real number.
\begin{cor}
 \noindent For any $x\in \R$, there exist infinitely many integers $p$ and $q > 0 $ such that
\begin{equation*}\label{side1}
\left\vert qx-p\right\vert<1/q.
\end{equation*}
\end{cor}

It is quite surprising that most metric theories on Diophantine approximation are intended to strengthen this corollary instead of Dirichlet's original theorem. Since Theorem \ref{Dir} was proved by a simple pigeon-hole argument, there should be a large room for improvement.

Let $\psi:[t_0, \infty)\to \mathbb R_+$ be a  non-increasing function with $t_0\geq 1$  fixed. Consider the set
\begin{equation*}
D(\psi):=\left\{x\in\mathbb R:  \begin{aligned} \exists\, N  \ {\rm such\  that\ the\ system}\ |qx-p|\, <\, \psi(t), \  
 |q|<t\  \\
  \qquad \text{has a nontrivial integer solution for all }t>N\quad
                           \end{aligned}
\right\}.
\end{equation*}
A real number $x$ will
be called
\emph{$\psi$-Dirichlet improvable} if $x\in D(\psi)$, and
elements of the complementary set, $D(\psi)^c$,  will be referred to as \emph{$\psi$-Dirichlet non-improvable} numbers.

In what follows we will use the notation $\psi_1(t) = 1/t$.
Dirichlet's theorem, 
 which essentially implies that $D(\psi_1) = \R$, is sharp in the following sense. {A classical result of Davenport \& Schmidt}   \cite{DaSc70}  implies that for any $\epsilon>0$, $D\big((1-\epsilon)\psi_1\big)$ is a subset of the union of $\Q$ and the set of badly approximable numbers.
Thus  $$\lambda\Big(D\big((1-\epsilon)\psi_1\big)^c\Big)=1.
$$
Moreover, as noticed by Kleinbock \& Wadleigh \cite{KlWad16},  $D(\psi)^c\ne \varnothing$
whenever $\psi$ is non-increasing and $$
 t\psi(t)<1  \ {\text{for all}}\ t\gg 1.
 $$

So it is natural to ask
how small are the corresponding sets, i.e.\ what is the size of the complement $D(\psi)^c$,  in the sense of measure/dimension, for functions $\psi$ which decay {slower} than $(1-\epsilon)\psi_1$ for any $\epsilon > 0$?
 Beyond some particular choices of $\psi$, nothing was known until recently, when Kleinbock \& Wadleigh \cite{KlWad16}
 proved 
 a dichotomy law on the Lebesgue measure of $D(\psi)$.

To state this result, as well as the main results of the present paper, we will introduce an {auxiliary} function\begin{equation}\label{twopsis}{
\Psi(t):=\frac{t\psi(t)}{1-t\psi(t)} = \frac{1}{1-t\psi(t)} - 1;
}\end{equation}
in what follows, $\psi$ and $\Psi$ will always be related by \eqref{twopsis}.

\begin{thm}[{\cite{KlWad16}, Theorem 1.8}]
\label{KW} Let $\psi:[t_0, \infty)\to \mathbb R_+$ be any non-increasing function such that  $\Psi$ as in \eqref{twopsis} (equivalently, the function $t\mapsto t\psi(t)$) is non-decreasing and $$t\psi(t)<1  \ \ \ for \ all \ \  t>t_0.$$
Then if
\begin{equation}\label{kwint}
\sum_{t}\frac{\log{\Psi}(t)}{t{\Psi}(t)}<\infty
\ \ (resp. =\infty)\end{equation}
then $$\lambda\big(D(\psi)^c\big)=0  \ \ ({\rm resp.} \ \lambda\big(D(\psi)\big)=0).$$ \end{thm}
For example, \begin{eqnarray*}
\lambda(D(\psi)^c)=\left\{
                     \begin{array}{ll}
                       0, & \hbox{{\rm if} \quad$\psi(t)=\frac{1}{t}\Big(1-\frac{1}{\log t (\log \log t)^{2+\epsilon}}\Big)$ {\rm for any} $\epsilon>0$;} \\
                      {\rm full}, & \hbox{{\rm if} \quad $\psi(t)=\frac{1}{t}\Big(1-\frac{1}{\log t (\log \log t)^{2}}\Big)$.}
                     \end{array}
                   \right.
\end{eqnarray*}

Notice that Theorem \ref{KW} fails to distinguish between  sizes of null sets. That is, if the approximating function $\psi$ decreases sufficiently {slowly}, Theorem \ref{KW} tells us that $D(\psi)^c$ is null but gives us no further information about the size of this set. For this purpose, Hausdorff measure and dimension are the appropriate tools. {One of our main results} is as follows:

\begin{thm}\label{dicor} Let $\psi$ be a non-increasing positive function with $t\psi(t)<1$ for all large $t$. Then for any $0\leq s<1$
\begin{equation}\label{hdsum} \H^s(D(\psi)^c)=\begin {cases}
 0 \ & {\rm if } \quad \sum\limits_{t} {t}\left(\frac{1}{{t^2\Psi({t})}} \right)^s
\, < \, \infty;  \\[2ex]
 \infty \ & {\rm if } \quad
 \sum\limits_{t} {t}\left(\frac{1}{{t^2\Psi({t})}} \right)^s \, = \, \infty.
\end {cases}\end{equation}

\end{thm}

Consequently, the Hausdorff dimension of the set $D(\psi)^c$ is given by
$$
\hdim D(\psi)^c=\frac{2}{2+\tau}, \ {\text{where}}\
\tau=\liminf_{t\to \infty}\frac{\log \Psi(t)}
{\log t}.
$$
As an example, $$
\hdim D(\psi)^c=\frac{2}{2+\tau},\    \ {\text{for}}\ \psi(t)=\frac{1-at^{-\tau}}{t}  \ (a>0, \tau>0).
$$

\begin{rem}\label{rem1}{\rm Here we remark that the condition $s<1$ is necessary.  $\H^1$ is the Lebesgue measure, which is the scope of Theorem \ref{KW}. The summability criterion that appears there does not agree with the one in 
Theorem \ref{dicor}:  indeed, when $s=1$,
the summand in \eqref{hdsum}
differs from that in \eqref{kwint} by a factor of $\log \Psi(t)$. This factor is not superfluous, as can be seen by taking $$\Psi(t)
= \log t\, (\log\log t)^2.$$
}\end{rem}

A natural generalization of the $s$-dimensional Hausdorff measure $\mathcal H^s$ is the $f$-dimensional Hausdorff measure $\mathcal H^f$ where $f$ is a {\sl dimension function}, that is an increasing, continuous function $f:\mathbb R_+\to \mathbb R_+$ such that $f(r)\to 0$ as $r\to 0$.
{We need to impose an additional technical condition on $f$: say that a dimension function $f$ is {\sl essentially sub-linear} if
\begin{equation}\label{ffm10}
\text{there exists }B>1\text{ such that } \limsup_{x\to 0}\frac{f(Bx)}{f(x)}<B.
\end{equation}
The above condition does not hold for $f(x) = x$ or $f(x)= x \log (1/x)$. However it is clearly  satisfied for the dimension functions $f(x)=x^s$ when $0\le s<1$.  Further, we remark that the essentially sub-linear condition is equivalent to the doubling condition but with  exponent $\alpha<1$.
  (A function $f$ is called {\sl doubling with exponent $\alpha$} if $f(cx)\ll c^\alpha f(x)$ for all  $x$ and all $c>1$.)

The following theorem   readily implies Theorem \ref{dicor}:}

\begin{thm}\label{dtthm} Let $\psi$ be a non-increasing positive function with $t\psi(t)<1$ for all large $t$, and let $f$ be {an essentially sub-linear} dimension function.
 Then
$$ \H^f\big(D(\psi)^c\big)=\begin {cases}
 0 \ & {\rm if } \quad \sum\limits_{t} {t}f\left(\frac{1}{{t^2\Psi({t})}} \right) \, < \, \infty; \\[2ex]
 \infty \ & {\rm if } \quad \sum\limits_{t} {t}f\left(\frac{1}{{t^2\Psi({t})}} \right)  \, = \, \infty.
\end {cases}$$
\end{thm}

\subsection{Continued fractions and improving Dirichlet's theorem
}
\
The starting point for the work of Davenport \& Schmidt \cite{DaSc70} and Kleinbock \& Wadleigh \cite{KlWad16} is an observation that Dirichlet improvability is equivalent to a condition on the growth rate of partial quotients. Let's recall this connection.

Every $x\in [0,1)$ has a  {\sl continued fraction expansion}, $$x
=\frac{1}{a_1(x)+\displaystyle{\frac{1}{a_2(x)+\displaystyle{\frac{1}{a_3(x)+\ddots}}}}}$$ where $a_1,a_2,\dots$ are positive integers called the {\sl partial quotients} of $x$. We write $x =[a_1(x),a_2(x),\dots]$ for short.   We also write $p_n/q_n = [a_1,...,a_n]$ ($p_n, q_n$ coprime) for the $n$'th  {\sl convergent} of $x$. The results of  \cite{DaSc70, KlWad16} rely crucially on the following observations,  proved
in  Lemma \ref{blah} below: \begin{align*}
x\in D(\psi) &\Longleftrightarrow |q_{n-1}x-p_{n-1}| \,<\, \psi(q_n)  \ {\text{for all}}\ n\gg 1 \\
&\Longleftrightarrow [a_{n+1}, a_{n+2},\dots]\cdot [a_n, a_{n-1},\dots, a_1]\, < \, \frac1{\Psi(q_n)}
  \ {\text{for all}}\ n\gg 1.\end{align*}
This leads to the following criteria for Dirichlet improvability.

\begin{lem}[\cite{KlWad16}, Lemma 2.2)]\label{kwlem}Let $x\in [0, 1)\smallsetminus\Q$, and let $\psi:[t_0, \infty)\to\R_+$ be non-increasing. Then
\begin{itemize}
\item [{\rm (i)}] $x$ is $\psi$-Dirichlet improvable 
 if $a_{n+1}(x)a_n(x)\, \le\,\Psi(q_n)/4$ for all sufficiently large $n$.
\item [{\rm (ii)}]$x$ is $\psi$-Dirichlet non-improvable
 if $a_{n+1}(x)a_n(x)\, >\, \Psi(q_n)$ for infinitely many~$n$.
\end{itemize}

\end{lem}
 As a consequence of this lemma, we have inclusions
\begin{equation}\label{l1.2}
G(\Psi) \subset D(\psi)^c\subset G(\Psi/4),\end{equation}
where
\begin{equation*}\label{gpsi}
G(\Psi):=\Big\{x\in [0,1): a_n(x)a_{n+1}(x)\,>\, \Psi\big(q_n(x)\big)  \ {\text{ for i.m.}}\ n\in \N\Big\}.
\end{equation*}

Our next theorem characterizes the $f$-dimensional Hausdorff measure of  sets $G(\Psi)$:

\begin{thm}\label{thm} Let $\Psi: [t_0, \infty)\to \R_+$ be a non-decreasing function and let $f$ be {an essentially sub-linear} dimension function.

Then
$$ \H^f\big(G(\Psi)\big)=\begin {cases}
 0 \ & {\rm if } \quad \sum\limits_{t} {t}f\left(\frac{1}{{t^2\Psi({t})}} \right) \, < \, \infty; \\[2ex]
 \infty \ & {\rm if } \quad \sum\limits_{t} {t}f\left(\frac{1}{{t^2\Psi({t})}} \right)  \, = \, \infty.
\end {cases}$$

\end{thm}

In view of the inclusion \eqref{l1.2},  Theorem \ref{thm} readily implies Theorem \ref{dtthm}.

\medskip

The structure of the paper is as follows.  In the next section  (\S \ref{S2}) we group together some basic definitions and concepts to which we will {appeal in} proving Theorem~\ref{thm}. The proof of Theorem \ref{thm} naturally splits into two parts, the divergence case and the convergence case, which we address in sections \S \ref{S3} and \S 4 respectively.  We conclude the paper with some related open questions.

\medskip

\noindent{\bf Acknowledgements.} M. Hussain would like to thank the members of the Department of Mathematics at Brandeis University for their hospitality. We would like to thank David Simmons   for some very useful comments on an earlier draft of this paper.

\section{Preliminaries and auxiliary results}\label{S2}

For completeness we give a very brief introduction to Hausdorff measures and dimension.  For further details we refer to the beautiful texts \cite{BeDo_99, F_14}.

\subsection{Hausdorff measure and dimension}\label{HM}\

Let $f$ be a dimension function and let
$E\subset \R^n$.
 Then, for any $\rho>0$ a countable collection $\{B_i\}$ of balls in
$\R^n$ with diameters $\mathrm{diam} (B_i)\le \rho$ such that
$E\subset \bigcup_i B_i$ is called a $\rho$-cover of $E$.
Let
\[
\H_\rho^f(E)=\inf \sum_i f\big(\mathrm{diam}(B_i)\big),
\]
where the infimum is taken over all possible $\rho$-covers $\{B_i\}$ of $E$. It is easy to see that $\H_\rho^f(E)$ increases as $\rho$ decreases and so approaches a limit as $\rho \rightarrow 0$. This limit could be zero or infinity, or take a finite positive value. Accordingly, the \textit{Hausdorff $s$-measure $\H^f$} of $E$ is defined to be
\[
\H^f(E)=\lim_{\rho\to 0}\H_\rho^f(E).
\]

It is easily verified that Hausdorff measure is monotonic and countably sub-additive, and that $\H^f(\varnothing)=0$. Thus it is an outer measure on $\R^n$.

In the case when $f(x)=x^s$ for some $s\geq 0$, we write $\mathcal H^s$ for $\mathcal H^f$. Furthermore, for any subset $E$ one can verify that there exists a unique critical value of $s$ at which $\H^s(E)$ `jumps' from infinity to zero. The value taken by $s$ at this discontinuity is referred to as the \textit{Hausdorff dimension of  $E$} and  is denoted by $\hdim E $; i.e.,
\[
\hdim E :=\inf\{s\ge 0:\; \H^s(E)=0\}.
\] When $s=n$,  $\H^n$ coincides with standard Lebesgue measure on $\R^n$.

Computing Hausdorff dimension of a set is typically accomplished in two steps: obtaining the upper and lower bounds separately.

Upper bounds often can be handled by finding appropriate coverings. When dealing with a limsup set, one 
 {usually applies} the Hausdorff measure version of the famous Borel--Cantelli lemma (see Lemma 3.10 of \cite{BeDo_99}):

\begin{pro}\label{bclem}
    Let $\{B_i\}_{i\ge 1}$ be a sequence of measurable  sets in $\R^n$ and suppose that for some dimension function $f$,  $$\sum_i f\big(\mathrm{diam}(B_i)\big) \, < \, \infty.$$ Then  $$\H^f(
    {\limsup_{i\to\infty}B_i})=0.$$
\end{pro}

\subsection{Continued fractions and Diophantine approximation}
 \

  Define the Gauss transformation $T:[0,1)\to [0,1)$  by
\[T(0):=0, \quad T(x):=\frac1x\ {\rm(mod}\ 1),\quad {\rm for} \ x\in (0, 1).\]
For $x \in [0,1)\smallsetminus \Q$ with continued fraction expansion $x= [a_1, a_2,\dots]$, as in section 1.2, we have $a_n(x)=\lfloor 1/T^{n-1}(x)\rfloor$ for each $n\ge 1.$ Recall the sequences $p_n= p_n(x)$, $q_n= q_n(x)$ also discussed in section 1.2. With the conventions
$p_{-1}=1, ~q_{-1}=0, ~p_0=0$ and  $~q_0=1$,  these sequences can be generated by the following recursive relations \cite{Khi_63}
\begin {equation}\label{recu}
p_{n+1}=a_{n+1}(x)p_n+p_{n-1}, \ \
q_{n+1}=a_{n+1}(x)q_n+q_{n-1},\ \  n\geq 0.
\end {equation}

Thus $p_n=p_n(x), q_n=q_n(x)$ are determined by the partial quotients $a_1,\dots,a_n$, so we may write $p_n=p_n(a_1,\dots, a_n), q_n=q_n(a_1,\dots,a_n)$. When it is clear which partial quotients are involved, we denote them by $p_n, q_n$ for simplicity.

For any integer vector $(a_1,\dots,a_n)\in \N^n$ with $n\geq 1$, write
\begin{equation}\label{cyl}
I_n(a_1,\dots,a_n):=\left\{x\in [0, 1): a_1(x)=a_1, \dots, a_n(x)=a_n\right\}
\end{equation}
for the corresponding `cylinder of order $n$', i.e.\  the set of all real numbers in $[0,1)$ whose continued fraction expansions begin with $(a_1, \dots, a_n).$

We will frequently use the following well known properties of continued fraction expansions.  They are explained in the standard texts \cite{IosKra_02, Khi_63, Khi_64}.

\begin{pro}\label{pp3} For any {positive} integers $a_1,\dots,a_n$, let $p_n=p_n(a_1,\dots,a_n)$ and $q_n=q_n(a_1,\dots,a_n)$ be defined recursively by \eqref{recu}. {Then:}
\begin{enumerate}[label={\rm (\subscript{\rm P}{\arabic*})}]
\item
\begin{eqnarray*}
I_n(a_1,a_2,\dots,a_n)= \left\{
\begin{array}{ll}
         \left[\frac{p_n}{q_n}, \frac{p_n+p_{n-1}}{q_n+q_{n-1}}\right)     & {\rm if }\ \
         n\ {\rm{is\ even}};\\
         \left(\frac{p_n+p_{n-1}}{q_n+q_{n-1}}, \frac{p_n}{q_n}\right]     & {\rm if }\ \
         n\ {\rm{is\ odd}}.
\end{array}
        \right.
\end{eqnarray*}
{\rm Thus, its length is given by} \begin{equation}\label{lencyl}
\frac{1}{2q_n^2}\leq |I_n(a_1,\ldots,a_n)|=\frac{1}{q_n(q_n+q_{n-1})}\leq \frac1{q_n^2},
\end{equation}
{\rm since} $$
 p_{n-1}q_n-p_nq_{n-1}=(-1)^n, \ {\rm for \ all }\ n\ge 1.
 $$

\item For any $n\geq 1$, $q_n\geq 2^{(n-1)/2}$.

\item $$\frac{q_{n-1}}{q_n}=[a_n, a_{n-1},
\dots, a_1].$$

\item $$
\big|q_{n-1}(x)x-p_{n-1}(x)\big|=\frac{1}{q_n(x)+T^n(x)\cdot q_{n-1}(x)}=\frac{1}{q_n(x)(1+T^n(x)  \cdot\frac{q_{n-1}(x)}{q_n(x)})},
$$

\item \begin{equation*}\label{p7}
\frac{1}{3a_{n+1}(x)q^2_n(x)}\, <\, \Big|x-\frac{p_n(x)}{q_n(x)}\Big|=\frac{1}{q_n(x)(q_{n+1}(x)+T^{n+1}(x)  q_n(x))}\, < \,\frac{1}{a_{n+1}q^2_n(x)}.
\end{equation*}
\end{enumerate}
\end{pro}

The next two theorems connect continued fractions to the theory of one-dimensional Diophantine approximation.

\begin{thm}[Lagrange] The convergents of $x$ are optimal rational approximations of $x$ in the sense that
$$\min_{q<q_n(x), p\in \N}|qx-p|=|q_{n-1}(x)x-p_{n-1}(x)|.$$
\end{thm}

\begin{thm}[Legendre]
\begin{equation}\label{Legendre}\Big|x-\frac pq\Big|<\frac1{2q^2}\Longrightarrow
 \frac pq=\frac{p_n(x)}{q_n(x)}\quad {\rm for\  some \ } n\geq 1.
\end{equation}\end{thm}

As mentioned in \S 1.2,  $x\in [0,1]{\smallsetminus \Q}$ is $\psi$-Dirichlet improvable if and only if the partial quotients of $x$ do not grow too quickly. We reproduce the proof by Kleinbock \& Wadleigh \cite[Lemma 2.1]{KlWad16} for completeness.

\begin{lem}\label{blah}Let $\psi:\N\to \R_+$ be non-increasing and suppose $t\psi(t)<1$ for all large $t$. Then
\begin{align*}
x\in D(\psi) &\Longleftrightarrow |q_{n-1}x-p_{n-1}| \,<\, \psi(q_n) \ {\rm for\ all}\ n\gg 1 \\
&\Longleftrightarrow [a_{n+1}, a_{n+2},\dots]\cdot [a_n, a_{n-1},\dots, a_1] \,> \frac1{\Psi(q_n)}
\  \ {\rm for\ all} \ n\gg 1.\end{align*} \end{lem}
\begin{proof}

For the first equivalence, if $x$ is in $D(\psi)$, then for all large $n$   there exist $p,q\in \N$ with $q<q_n$ such that $$
|qx-p| \,<\,  \psi(q_n).
$$Then by Lagrange's theorem, $$
|q_{n-1}x-p_{n-1}|\le |qx-p| \,<\,  \psi(q_n).
$$
On the other hand, assume that $$|q_{n-1}x-p_{n-1}| \,<\,  \psi(q_n), \ {\text{for all}}\ n\gg 1.$$
Then for any  large $t$ 
 let $n$ be such that $q_{n-1}<t\le q_n$. 
 By the monotonicity of $\psi$, one has $$
|q_{n-1}x-p_{n-1}|  \,<\,   \psi(q_n)\le \psi(t).
$$
The second equivalence follows from \eqref{twopsis} and from (P$_3$) and (P$_4$) in Proposition~\ref{pp3} via a simple computation. \end{proof}

\subsection{Classical Jarn\'ik's  theorem for well approximable points}\label{aux}
\
Define \begin{equation*}
\mathcal K(\Psi):=\left\{x\in[0,1): \left|x-\frac pq\right|<\frac{1}{q^2\Psi(q)} \ {\rm for \ i.m. \ } (p, q)\in \Z\times \N \right\}.
\end{equation*}
This is just the set of {$\Phi$-approximable numbers if we take $\Phi(q)=\frac{1}{q\Psi(q)}$.

An elegant  zero-infinity  law for the Hausdorff measure of the sets $\mathcal K(\Psi)$ is due to Jarn\'ik \cite{Jar31}. 
We will need the following refined version:

\begin{thm}\label{Jarnik}
Let $\Psi$ be a non-increasing function, and let $f$ be a dimension function satisfying 
the following properties:
\begin{equation}\label{f2infty}
\lim_{x\to 0}\frac{f(x)}{x}=\infty,
\end{equation} 
and \begin{equation}\label{f2}
 \exists\, C \ge 1  \text{ such that } \frac{f(x_2)}{x_2}\le C\frac{f(x_1)}{x_1}  {\text{ whenever}}\ x_1< x_2\ll 1.
\end{equation} 
Then
\begin{equation}\label{sum1} \H^f\big(\mathcal K(\Psi)\big)=\begin {cases}
 0 \ & {\rm if } \quad
 \sum\limits_{t} {t}f\left(\frac{1}{t^2\Psi(t)}\right) \, < \, \infty; \\[2ex]
 \infty \ & {\rm if } \quad
 \sum\limits_{t} {t}f\left(\frac{1}{{t^2\Psi({t})}}\right)\, = \, \infty.
\end {cases}\end{equation}
\end{thm}

\begin{proof}[Sketch of Proof]
Note that   the original formulation of Jarn\'ik (see e.g.\ \cite[\S 1.1]{BDV06}) assumes condition \eqref{f2} with $C = 1$, that is, the monotonicity of the function $x\mapsto \frac{f(x)}x$.
However it is not hard to see that it can be replaced with ``quasi-monotonicity" as in \eqref{f2}. For example, a modern proof of Jarn\'ik's theorem, due to Beresnevich \& Velani \cite[Theorem 2]{BV06},
 is given by a combination of {Khintchine's} classical theorem \cite{Khi_64} and  the Mass Transference Principle. One can observe, however, that in proving the
  latter theorem the monotonicity assumption on the function $x\mapsto \frac{f(x)}x$ is only used in the last step of the proof,  see the last inequality in the formula (29) {in \cite{BV06}}. {The latter} still works if 
\eqref{f2} is used instead.
\end{proof}}


\section{Proof of Theorem \ref{thm}: the divergence case}\label{S3}


Recall that $$
G(\Psi)=\Big\{x\in [0,1):  a_{n}(x)a_{n+1}(x)\ge \Psi(q_n) \ \ {\rm for \ i.m.}\ n\in \N\Big\}.
$$
To prove the divergence case  of Theorem \ref{thm} we first notice an obvious inclusion
$$G(\Psi)\supset \left\{x\in[0, 1): a_{n+1}(x)\,>\, \Psi(q_n)  \ \ {\rm for \ i.m.\ }n\in\N\right\}=:\,G_1(\Psi).$$

We can assume that $ \Psi(t)\ge 1$ for all $t\gg 1$. Otherwise, $\Psi(t)<1$ for all  large $t$ since we have assumed $\Psi$ to be non-decreasing. Then it is obvious that $G_1(\Psi)$, and thus $G(\Psi)$, contains all irrational numbers in $[0,1]$, and that the sum in Theorem \ref{thm} diverges.  

It is well known that $G_1(\Psi)$ contains $\mathcal K(3\Psi)$. The proof is rather short, so we include it for completeness.  Indeed, if there are infinitely many $(p,q)$ with  $$
|x-p/q|<\frac{1}{{3\Psi}(q)q^2} < \frac{1}{2q^2},
$$ then, by Legendre's theorem,  $$
\frac{p}{q}=\frac{p_n(x)}{q_n(x)}\ \ {\text{for some}}\ n\ge 1.
$$
Since $p_n, q_n$ are coprime, we must have $q_n\le q$. So, by the monotonicity of $\Psi$, $$
\big|x-\frac{p_n}{q_n}\big|=\big|x-\frac{p}{q}\big|<\frac{1}{{3\Psi}(q)q^2}\le \frac{1}{{3\Psi}(q_n)q_n^2}.
$$ On the other hand, in view of ($P_5$), $$
\big|x-\frac{p_n}{q_n}\big|\ge  \frac{1}{3a_{n+1}q_n^2}.
$$ This implies $
a_{n+1}> 
{\Psi}(q_n),
$ for infinitely many $n$, verifying the claim.

{Thus by Theorem \ref{Jarnik}} 
{one} will have
 $$
 \mathcal{H}^f\big(G(\Psi)\big)\ge \mathcal{H}^f\big(G_1(\Psi)\big)\ge \mathcal{H}^f\big(\mathcal K(3\Psi)\big)=\infty
 $$ {whenever {one} can show that the dimension function $f$ satisfies  conditions  \eqref{f2infty}  and (\ref{f2}), and that \begin{equation}\label{sum2}
  \sum\limits_{t} {t}f\left(\frac{1}{{3t^2\Psi({t})}} \right)
 = \infty.\end{equation} This is done via the following lemma. \begin{lem} Let $f$ be an essentially sub-linear   dimension function. Then both \eqref{f2infty}  and \eqref{f2}  hold.

 \end{lem}

 \begin{proof} 
Condition  \eqref{ffm10} implies that there exist $\epsilon,\delta > 0 $ and $B > 1$ such that
 \begin{equation}\label{f3}0 < x < \delta\quad  \Longrightarrow \quad \frac{f(Bx)}{f(x)} < B - \epsilon.\end{equation}
Therefore for some $0 < x_0 < \delta$ and all $n\ge 1$ one has
$$
f(x_0/B^n) > f(x_0)/(B-\epsilon)^n \quad  \Longleftrightarrow \quad \frac{f(x_0/B^n)}{x_0/B^n} > \left(\frac B{B-\epsilon}\right)^n\frac{f(x_0)}{x_0}.
$$
This shows that
$f(x)/x\to \infty$ as $x\to 0$. 
As for (\ref{f2}), let  $x_1<x_2<\delta$. Assume $$B^{-k}\le x_2 < B^{-k+1}, \ B^{-\ell}\le x_1< B^{-\ell+1},\ {\text{with}}\ k\le \ell.$$
Then
\begin{align*}
\frac{f(x_2)}{f(x_1)} \cdot  \frac{x_1}{x_2} \le \frac{f(B^{-k+1})}{f(B^{-\ell})} \cdot \frac{B^{-\ell+1}}{B^{-k}}\le (B-\epsilon)^{\ell-k+1}\cdot B^{-\ell+k+1}\le B^2.
\end{align*} Therefore, $$
\frac{f(x_2)}{x_2} \le B^2 \cdot \frac{f(x_1)}{x_1},
$$
and (\ref{f2}) follows.
 \end{proof}

Finally, notice that (\ref{sum2}) is equivalent to  (\ref{sum1})  as  $f$ is increasing and, by (\ref{f3}),
has doubling property.
This settles the divergence case of Theorem \ref{thm}.}

 \section{Proof of Theorem \ref{thm}: the Convergence Case}

The set $G(\Psi)$ can be written in terms of the following basic cylinders:
$$
G(\Psi)=\bigcap_{N=1}^{\infty}\bigcup_{n=N}^{\infty}\bigcup_{a_1,\dots,a_n}J_n(a_1,\dots,a_n),
$$ where $$
J_n(a_1,\dots,a_n):=\bigcup_{a_{n+1} > \frac{\Psi(q_n)}{a_n}}I_{n+1}(a_1,\dots,a_n,a_{n+1}).
$$ Using (P$_1$) in Proposition \ref{pp3} and the recursive relation \eqref{recu},
the diameter of $J_n(a_1,\dots,a_n)$ can be bounded {as follows}:
 \begin{align*}
|J_n(a_1,\dots,a_n)|&=\sum_{a_{n+1} > \frac{\Psi(q_n)}{a_n}}\left|\frac{a_{n+1}p_n+p_{n-1}}{a_{n+1}q_n+q_{n-1}}-\frac{(a_{n+1}+1)p_n+p_{n-1}}{(a_{n+1}+1)q_n+q_{n-1}}\right|\\
&\le \left|\frac{\frac{\Psi(q_n)}{a_n}p_n+p_{n-1}}{\frac{\Psi(q_n)}{a_n}q_n+q_{n-1}}-\frac{p_n}{q_n}\right|=
\frac{1}{\left(\frac{\Psi(q_n)}{a_n}q_n+q_{n-1}\right)q_n}\\
&\le \frac{1}{\Psi(q_n)q_{n-1} q_n}.\end{align*}



Therefore, $$
\mathcal{H}^f\big(G(\Psi)\big)\le \liminf_{N\to \infty}\sum_{n\geq N}\sum_{a_1,...a_n}f\left(\frac{1}{\Psi(q_n)q_{n-1}q_{n}}\right).
$$

So, the remaining task is to estimate the summation over $$
\mathcal{A}_N:=\{(a_1,\dots,a_n): n\ge N\}.$$ We first partition {$\mathcal A_N$} by introducing a function $
g_N: \mathcal{A}_N\to \N^2$, defined as $$ {g_N}\big((a_1,\dots,a_n)\big)=\big(q_{n-1}(a_1,\dots,a_n), q_{n}(a_1,\dots,a_n)\big),
$$
where $q_n$ is defined by the recurrence relation \eqref{recu}. The following observations can readily be verified.

\smallskip

\begin{enumerate}[label=(\subscript{\rm Ob}{\arabic*})]

\item The function ${g_N}$ is two-to-one. This is because a rational number has two continued fraction  representations.
  More precisely, continued fraction expansions of rational numbers are not allowed to terminate in $1$; if $p/q$ has continued fraction expansion $[b_1,\dots,b_k]$, one must have $b_k\ge 2$. However it is also true that $$
    p/q=[b_1,\dots,b_{k}-1, 1].
    $$

    Now fix a positive integer vector $(p,q)$ such that $p/q$ has expansion $[b_1,\dots,b_{k}]$, and assume that ${g_N}\big((a_1,\dots,a_n)\big)=(p,q)$. Then by ($P_3$) in Proposition \ref{pp3},
$$
\frac{p}{q}=\frac{q_{n-1}(a_1,\dots,a_n)}{q_n(a_1,\dots,a_n)}=[a_n, a_{n-1},\dots,a_1],
$$ This gives a continued fraction representation of $p/q$. So, $$
(a_n,\dots,a_1)=(b_1,\dots,b_k)\ {\text{or}}\ (b_1,\dots,b_k-1,1).
$$ On the other hand, it is straightforward to check that $$
{g_N}\big((b_k,\dots, b_1)\big)={g_N}\big((1, b_k-1,\dots, b_1)\big)=(p,q).
$$

\item The range of ${g_N}$ is a subset of $$
\mathcal{C}_N:=\Big\{(p,q)\in \N^2: gcd(p,q)=1, \ 1\le p\le q,\  q\ge 2^{(N-1)/2}\Big\}.
$$
\item The following is a partition {of $\mathcal{A}_N$:} $$
\mathcal{A}_N=\bigcup_{(p,q)\in {\mathcal{C}_N}}g_N^{-1}(p,q).
$$

\end{enumerate}

As a result,  for a dimension function $f$, we have
\begin{align}\label{nick}
\mathcal{H}^f\big(G(\Psi)\big)\le& \liminf_{N\to \infty} \sum_{(p,q)\in \mathcal{C}_N}\sum_{g_N^{-1}(p,q)}f\left(\frac{1}{q_{n-1}q_n\Psi(q_n)}\right)\notag\\
\le & 2\liminf_{N\to \infty}\sum_{(p,q)\in \mathcal{C}_N}f\left(\frac{1}{pq\Psi(q)}\right) \notag\\ \le& 2\liminf_{N\to \infty}\sum_{q\ge 2^{(N-1)/2}}\sum_{1\le p\le q}f\left(\frac{1}{pq\Psi(q)}\right).\end{align}

It can be seen that if
 \begin{equation}\label{eqn1}\sum_{q=1}^\infty\sum_{1\le p\le q}f\left(\frac{1}{pq\Psi(q)}\right)<\infty,\end{equation}
 then it readily follows from Proposition \ref{bclem} that $\mathcal H^f\big(G(\Psi)\big)=0.$

To complete the proof of the convergence case, it remains to show that
\begin{equation*}
\sum_{q=1}^{\infty}\sum_{1\le p\le q}f\Big(\frac{1}{pq\Psi(q)}\Big)\quad  {\rm and} \quad \sum_{q=1}^{\infty}qf\Big(\frac{1}{q^2\Psi(q)}\Big)
\end{equation*}
 have the same convergence and divergence property.  It is straightforward to establish that  $$
\sum_{q=1}^{\infty}qf\Big(\frac{1}{q^2\Psi(q)}\Big)=\infty \Longrightarrow \sum_{q=1}^{\infty}\sum_{1\le p\le q}f\Big(\frac{1}{pq\Psi(q)}\Big)=\infty,
$$ since by the increasing property of $f$, $$
\sum_{1\le p\le q}f\Big(\frac{1}{pq\Psi(q)}\Big)\ge \sum_{1\le p\le q}f\Big(\frac{1}{q^2\Psi(q)}\Big)=qf\Big(\frac{1}{q^2\Psi(q)}\Big).
$$

So, to finish the proof of the convergence case, it remains to verify that \begin{equation}\label{ff7}
\sum_{q=1}^{\infty}qf\Big(\frac{1}{q^2\Psi(q)}\Big)<\infty \Longrightarrow \sum_{q=1}^{\infty}\sum_{1\le p\le q}f\Big(\frac{1}{pq\Psi(q)}\Big)<\infty.
\end{equation}

Clearly this implication is not true when $f(x)=x^s$ for $s=1$. This is also not true for {$f(x)=x \log(1/x)$, see 
Example \ref{xlogx}
below.} Therefore, a natural question is to classify dimension functions $f$ for which the assertion \eqref{ff7} holds. It turns out that this assertion is satisfied for essentially sub-linear    dimension functions.

 \begin{pro}\label{p1}
{Let $f$ be an essentially sub-linear   dimension function.}
Then the assertion in \eqref{ff7} is true.
\end{pro}

{\em{Proof}}. 
Fix $b<B$ such that \begin{equation}\label{ff8}
\frac{f(Bx)}{f(x)}<b  \ {\text{when}} \ x<x_0. 
\end{equation}  Let $$q_0 {>} \max\{x_0^{-1}, B\},$$ which is designed
{so that} the inequality (\ref{ff8}) {can be utilized} later.

Now for each $q\ge q_0$ we estimate the inner summation in the series $$
\sum_{q=1}^{\infty}\sum_{1\le p\le q}f\Big(\frac{1}{pq\Psi(q)}\Big).
$$
Let $t$ be the integer such that $B^{t-1}\le  q< B^{t}$. Then \begin{align*}
\sum_{1\le p\le q}f\Big(\frac{1}{pq\Psi(q)}\Big)&=\sum_{k=1}^t\sum_{B^{-k} q< p\le B^{-k+1}q}f\Big(\frac{1}{pq\Psi(q)}\Big)\\
&\le \sum_{k=1}^tB^{-k+1}qf\Big(\frac{B^{k}}{q^2\Psi(q)}\Big){=}B\sum_{k=1}^tC_k,
\end{align*}
{where $C_k := B^{-k}qf\Big(\frac{B^{k}}{q^2\Psi(q)}\Big)$}. Notice that for any $k<t$, \begin{align*}\frac{C_{k+1}}{C_k}=\frac{B^{-k-1}qf\Big(\frac{B^{k+1}}{q^2\Psi(q)}\Big)}{B^{-k}qf\Big(\frac{B^{k}}{q^2\Psi(q)}\Big)}:=\frac{f(Bx)}{Bf(x)}<\frac{b}{B},
\end{align*} since $$
x:=\frac{B^{k}}{q^2\Psi(q)}\le \frac{1}{q\Psi(q)}<x_0.
$$
Thus for any $1\le k\le t$, $$
{C_k}\le \left(\frac{b}{B}\right)^{k-1}C_1.$$
As a result, $$
\sum_{k=1}^tC_k\le \sum_{k=1}^{t}\left(\frac{b}{B}\right)^{k-1}C_1\le cC_1=\frac{c}{B}qf\Big(\frac{B}{q^2\Psi(q)}\Big)\le \frac{cb}{B}\cdot qf\Big(\frac{1}{q^2\Psi(q)}\Big).
$$

In summary, we have proved $$
\sum_{1\le p\le q}f\Big(\frac{1}{pq\Psi(q)}\Big)\le cb\cdot qf\Big(\frac{1}{q^2\Psi(q)}\Big).
$$ So, the desired assertion follows, and the proof of  Theorem \ref{thm} is thus completed.



\section{Final comments and open problems}\label{S4}
Our approach to the problems discussed in this paper is reasonably general. Although, together with the results of \cite{KlWad16}, we  {have almost} complete information {on the size of sets} of Dirichlet non-improvable real numbers in {the one-dimensional} setting, there  are still some open problems which could improve the current state of knowledge.
We list some of them here.

\subsection{} \label{notesl} As stated earlier, the 
 functions $f(x)=x$ and $f(x)=x\log(1/x)$ are not essentially sub-linear. For these particular examples, {the argument  of this paper only leads} to a weaker/incomplete characterization {of} $\mathcal{H}^f(D^c(\psi))$. For clarity we {emphasize it in the} two examples below.

\noindent
\begin{exa}[{{\rm The  case}  $f(x)=x$}] {\rm Notice that the estimate \eqref{nick}  is valid for any dimension function. Hence when we substitute $f(x)=x$, we get
$\mathcal{H}^1\big(G(\Psi)\big)=0$ if
  \begin{equation}\label{weak}
 \sum_{t}
 \frac{\log t}{t\Psi(t)}<\infty.
\end{equation}
This is, however, weaker than what is proved in \cite{KlWad16}. Indeed, it is shown in \cite[Corollary 3.7]{KlWad16} that the Lebesgue measure of $G(\Psi)$ is zero/full if the series
\begin{equation}\label{lebesgue}
\sum_t \frac{\log {\Psi}(t)}{{t\Psi}(t)}\end{equation} converges/diverges. It is not hard to see that assuming that $\Psi: \N \rightarrow \R_+$ is non-decreasing, the convergence condition \eqref{weak} implies that the series \eqref{lebesgue} converges.
The converse implication  is not always true.   One counterexample (similar to the function in {Remark \ref{rem1}})
 is given by
$$
\Psi(q)=\log q (\log\log q)^{2+\epsilon}, \ \ \epsilon>0.
$$
For this example,
\begin{align*}
&\sum_{q=1}^{\infty}\frac{\log q}{q\log \Psi}\asymp \sum_{q=1}^{\infty}\frac{1}{q(\log\log q)^{2+\epsilon}}=\infty;\\
&\sum_{q=1}^{\infty}\frac{\log \Psi}{q\Psi(q)}\asymp\sum_{q=1}^{\infty}\frac{1}{q\log q (\log\log q)^{1+\epsilon}}<\infty.
\end{align*}}
\end{exa}


\begin{exa} 
[{{\rm The  case}  $f(x)=x\log(1/x)$}]\label{xlogx}
{\rm Similar to the above example, for this choice of the dimension function $f$ our argument again gives an incomplete result. Substituting $f$ into  \eqref{sum2} for the divergence part and into \eqref{nick}  for the convergence part, we get that \begin{equation}\label{s4} \H^f\big(\mathcal D^c(\psi)\big)=\begin {cases}
 0 \ & {\rm if } \quad
 \sum\limits_{t} \frac{\log^2 t}{t\Psi(t)} \, < \, \infty, \\[2ex]
 \infty \ & {\rm if } \quad
 \sum\limits_{t} \frac{\log t}{{t\Psi({t})}}\, = \, \infty.
\end {cases}\end{equation}}
\end{exa}

David Simmons communicated to us that, by using a slightly different method, 
the convergent case of the dichotomy \eqref{s4} can be improved to   \[\sum_t \frac{\log t \log(\Psi(t))}{t \Psi(t)} \, < \, \infty.\]
His method is applicable to a certain class of 
not essentially sub-linear dimension functions.
Obtaining an analogue of Theorem \ref{thm} for all 
dimension functions remains an open problem.

\subsection{} Recall that, to prove the divergence case of Theorem \ref{thm},  we first {showed}
that the set of well approximable points with a dilated approximating function is contained in the set $G(\Psi)$, i.e.
$$\mathcal K(3\Psi)\subset G(\Psi).$$
So that, trivially, we have the divergence case of classical Jarn\'ik's theorem  at our disposal.
%
 This raises an immediate question:

\begin{prob}\label{sizedifference}
 How big is the set $G(\Psi)\smallsetminus \mathcal K(3\Psi)$?
  \end{prob}

We are in a position to claim the solution to this problem{\color{red},} which  will be the subject of a forthcoming article. To be precise, using a Cantor-type construction and the mass distribution principle \cite{F_14} we can prove
{the following result:}
\begin{thm} Let ${C}>0$. Then for $\Psi$ as above we have $$
\hdim \Big(G(\Psi)\smallsetminus \mathcal K({C}\Psi)\Big)=\frac{2}{\tau+2}, \text{ where } \ \tau=\liminf_{q\to \infty}\frac{\log \Psi(q)}{\log q}.
$$
  \end{thm}

\subsection{} Notice that the set $G(\Psi)$ contains those real numbers whose approximation properties are tied up with the growth of the product of  
{pairs of consecutive partial quotients}. However the growth rate is a function of the denominator of the rational approximates. What happens if the approximating function is just a function of the index $n$, i.e.\ if one considers the set
\[F(\varphi):=\left\{x\in [0, 1): a_n(x)a_{n+1}(x)\geq \varphi(n) \ \ {\rm for \ infinitely \ many \ }n\in\N\right\},\]
where $\varphi:\N\to\R_+$ is an arbitrary positive function? Here we have chosen the function $\varphi$ to distinguish it from the the function $\Psi$ which is biased to the relation \eqref{twopsis}.

\begin{prob}\label{sizeF} How big is the set $F(\varphi)$?
\end{prob}

{Note that solving}  Problem \ref{sizeF} in terms of Hausdorff measure and dimension  would not imply Theorem 
 \ref{thm}. Therefore, although interesting on its own, {a} solution to this problem {would} not contribute much to our understanding of the size of   {sets} of Dirichlet non-improvable real numbers in terms of Hausdorff measure or Hausdorff dimension. With regards to the Lebesgue measure of this set, it has been addressed already in \cite[Theorem 3.6]{KlWad16}.

It is worth pointing out that the Hausdorff dimension of the set \[
  \left\{x\in [0, 1): a_n(x)\geq \varphi(n) \ \ {\rm for \ infinitely \ many \ }n\in\N\right\}\subset F(\varphi)\]
has already been 
{computed} by Wang  \& Wu \cite{WaWu08}. It is expected that the Hausdorff dimension of   $F(\varphi)$ can be determined using similar methods. 
In fact, we have {made substantial} progress towards this problem, 
which will be the subject of a forthcoming article.



{}



\subsection{}
Another natural question is {a} generalization of the results of this paper and{ those} of \cite{KlWad16} to {the {higher-dimensional} setting}. For further details about {such a} generalization in {the} Lebesgue measure context we refer the reader to the last section of \cite{KlWad16}. {Once Lebesgue measure is settled for higher dimensions, one may ponder Hausdorff measure in this setting.} 

\def\cprime{$'$} \def\cprime{$'$} \def\cprime{$'$} \def\cprime{$'$}
  \def\cprime{$'$} \def\cprime{$'$} \def\cprime{$'$} \def\cprime{$'$}
  \def\cprime{$'$} \def\cprime{$'$} \def\cprime{$'$} \def\cprime{$'$}
  \def\cprime{$'$} \def\cprime{$'$} \def\cprime{$'$} \def\cprime{$'$}
  \def\cprime{$'$} \def\cprime{$'$} \def\cprime{$'$} \def\cprime{$'$}
  \def\cprime{$'$}

\end{document}